%% file: IQC_necessity_arXiv_v2.tex
\documentclass[10pt,final]{IEEEtran}

\usepackage[english]{babel}                             
\usepackage{amsthm, amsmath, amssymb, amsfonts}         
\usepackage{graphicx}                                   
\usepackage{color}    
\usepackage{mathrsfs}     
\usepackage{hhline}
\usepackage{cancel}                             
\graphicspath{{images/}}

\newtheorem{theorem}{Theorem}

\newtheorem{corollary}{Corollary}
\newtheorem{proposition}{Proposition}
\newtheorem{lemma}{Lemma}
\newtheorem{definition}{Definition}

\newtheorem{remark}{Remark}

\newcommand{\TwoOne}[2]
{\begin{bmatrix}
{#1} \\
{#2}
\end{bmatrix}
}

\newcommand{\TwoTwo}[4]
{\begin{bmatrix}
{#1} & {#2} \\
{#3} & {#4}
\end{bmatrix}
}
\newcommand{\STwoOne}[2]
{\left[\begin{smallmatrix}
{#1} \\
{#2}
\end{smallmatrix}\right]
}
\newcommand{\STwoTwo}[4]
{\left[\begin{smallmatrix}
{#1} & {#2} \\
{#3} & {#4}
\end{smallmatrix}\right]
}

\newcommand{\LT}{\mathscr{L}_{\rm 2}}
\newcommand{\LTE}{\mathscr{L}_{\rm 2e}}
\newcommand{\jw}{\lambda}
\newcommand{\jj}{\mathrm{j}}
\newcommand{\bS}{\mathbb{S}}
\newcommand{\pS}{\partial\mathbb{S}}
\newcommand{\Real}{\ensuremath{\mathbb R}}
\newcommand{\Complex}{\ensuremath{\mathbb C}}
\newcommand{\Hinf}{\mathbb{H}_{\infty}}
\newcommand{\RHinf}{\mathbb{RH}_{\infty}}
\newcommand{\He}{\mathrm{He}}
\newcommand{\esssup}{\mathop{\mathrm{ess\,sup}}}
\newcommand{\cont}{\mathbb{F}}

\begin{document}

\title{Converse Theorems for \\ Integral Quadratic Constraints\thanks{ This work was supported in part by the Ministry of Science and 
Technology of Taiwan. \emph{(Corresponding author: Chung-Yao Kao)} }  }

\author{Sei~Zhen~Khong~and~Chung-Yao~Kao \thanks{
S. Z. Khong was with the Department of Electrical and Electronic Engineering, The University of Hong Kong, and now is an independent researcher.
(email: szkhongwork@gmail.com)}\thanks{
C.-Y. Kao is with the Depart. of Electrical Engineering, National Sun Yat-Sen University, Kaohsiung,
Taiwan. (e-mail: cykao@mail.ee.nsysu.edu.tw)}
} 

\maketitle

\begin{abstract}
  A collection of converse theorems for integral quadratic constraints (IQCs) is established for linear time-invariant systems. It is demonstrated
  that when a system interconnected in feedback with an arbitrary system satisfying an IQC is stable, then the given system must necessarily satisfy
  the complementary IQC. These theorems are specialized to derive multiple versions of converse passivity results. They cover standard notions of
  strict passivity as well as passivity indices that characterize the tradeoffs between passivity surplus and deficit. Converse frequency-weighted
  small-gain and passivity theorems are also established.
\end{abstract}

\emph{keywords}:  
  robust stability, integral quadratic constraints, passivity, small-gain

\section{Introduction}

Integral quadratic constraints (IQCs) are a well-established tool for robustness analysis of feedback interconnected systems from the 
input-output perspective~\cite{MJKR10}, and stand as a parallel to the state-space methods based on dissipativity~\cite{Wil72, LBEM13, AMP16}. 
The IQC analysis, as it was first introduced in~\cite{MegRan97}, provides a sufficient condition under which robust closed-loop stability of
nonlinear systems can be certified. It generalizes the standard small-gain and passivity results, besides allowing the use of dynamical 
multipliers to reduce conservatism. Despite its immense versatility, necessity of the IQC condition has rarely been studied in the literature. 
The objective of this paper is to establish certain converse IQC results in the linear time-invariant (LTI) setting, so as to further 
substantiate and promote the utility of IQCs in robust stability analysis.

The converse IQC results in this paper are concerned with uncertainties described by IQCs. Specifically, it is shown in 
Section~\ref{sec: uncertainty}
that if a feedback interconnection of a given system and any uncertain system satisfying an IQC is (robustly) stable, then the given system 
must satisfy the complementary IQC. The proof relies on the multiplier admitting a specific $J$-spectral factorization~\cite{GGLD90} and the
construction of a destablizing open-loop component in the well-known small-gain theorem~\cite{ZDG96}. In effect, the results demonstrate that 
IQC analysis is not conservative if the feedback interconnection is required to be robust against all the uncertainties as characterized by 
a specific IQC.

By specializing the IQC results to particular forms of multipliers, various versions of converse passivity theorems are derived in
Section~\ref{sec: passive}. They cover both input and output strict passivity~\cite{Sch17}, as well as compensation for the lack of 
passivity in one subcomponent with excess passivity in another, as elegantly quantified by the notion of passivity 
indices~\cite{Vid81,BaoLee07,KMXGA14,Sch17}. While
converse passivity theorems have been investigated in the time-varying setting in~\cite{KhoSch18}, they cannot be used to recover the 
LTI results in this paper. In particular, the set of (nonlinear time-varying) uncertainties in~\cite{KhoSch18} is larger than that considered 
in this paper, which is taken to be LTI. This gives rise to different ramifications in the sufficiency and necessity proofs of the results. Furthermore, unlike the necessity
proofs in~\cite{KhoSch18}, which rely on the S-procedure lossless theorem~\cite{MegTre93}, the ones in this paper are constructive. It is also
noteworthy that the single-input-single-output version of the converse passivity theorems in this paper has been considered in~\cite{colgate-hogan}
using arguments from the Nyquist stability theory. The latter paper is motivated by applications in robotics, as is further elaborated
in~\cite{Str15}. Specifically, in order to guarantee the stability of a controlled robot interacting with a passive but otherwise unknown environment,
the converse passivity theorem dictates that the robot must exhibit some form of strict passivity as seen from its interaction ports.

In Section~\ref{sec: inf_dim}, we establish a generalization of a converse IQC result to infinite-dimensional multipliers. This is 
subsequently employed to prove converse frequency-weighted small-gain and passivity theorems. They naturally extend the standard small-gain 
and passivity results through the use of frequency weights so as to reduce conservatism in robustness analysis.

The results presented in Section~\ref{sec: passive} can be proven via a path that is technically more direct and of a similar
spirit~\cite{KKS_CDC19}.
The approach adopted in~\cite{KKS_CDC19} focuses only on obtaining the converse passivity results, while the theorems presented 
in Sections~\ref{sec: uncertainty} to~\ref{sec: inf_dim} are applicable in a much broader range, where uncertain LTI systems are
characterized by general quadratic forms that may even be defined by infinite-dimensional multipliers.

Concluding remarks are provided and several future research directions discussed in Section~\ref{sec: conclusion}. The next section presents the
notation and mathematical preliminaries used throughout the paper.

\section{Notation and Preliminaries}

The results described in this paper hold in both the continuous-time (CT) and 
discrete-time (DT) domains. Thus, notation is selected to facilitate
the development that respects this fact. 

$\Real$~($\Complex$), $\Real^n$~($\Complex^n$), $\Real^{n\times m}$~($\Complex^{n\times m}$) 
denote the sets of real (complex)
numbers, $n$-dimensional real (complex) vectors, and $n\times m$ real (complex)
matrices, respectively. Let the extended real set
$\bar{\mathbb{R}}:=\Real\cup\{\pm\infty\}$ and the nonnegative orthant
of $\Real^n$ be denoted by $\Real^n_+$. The sets of integers and non-negative integers
are denoted as $\mathbb{Z}$ and $\mathbb{Z}_+$, respectively. The so-called 
``stability region'' is denoted by $\bS$, which represents the open left-half of the complex plane
for the CT case, and the open unit disk for the DT case. The boundary of $\bS$ (i.e., 
``stability boundary'') is denoted by $\pS$, which is the imaginary axis for the CT case and 
the unit circle for the DT case. The ``instability region'' is the complement of $\bS$, 
denoted as $\bS^c$. 

Given a matrix $M$, the transpose and conjugate transpose are denoted by $M^T$ and $M^*$, respectively.  The maximum singular value of $M$ is denoted
by $\bar{\sigma}(M)$.  For a square matrix $M$, the Hermitian part of $M$ (scaled by a factor of 2) is denoted by $\He(M):=M+M^*$. The notation $M>0$
($M \ge 0$) means that the matrix $M$ is positive definite (positive semi-definite).  The $n$-dimensional identity matrix and $n\times m$ zero matrix
are denoted by $I_n$ and $\mathbf{0}_{n \times m}$, respectively. The subscripts of these matrices are dropped when their dimensions are clear from
the context.

We use $\LT^n$ to denote the space of $\Real^n$-valued, CT square-integrable
functions on $\Real_+$, or DT square-summable functions on $\mathbb{Z}_+$, 
with the usual norm and inner
product denoted by $\|\cdot\|_{\LT}$ and
$\langle\cdot,\cdot\rangle_{\LT}$, respectively.  The superscript 
is dropped when the dimension is evident from the
context. The extended $\LT^n$ space is denoted as $\LTE^n$. This
consists of functions $f$ that satisfy $P_T f\in\LT^n$, for all $T>0$,
where $P_T$ denotes the truncation operator defined as:
\[
(P_Tf)(t) = \begin{cases}
f(t) & \text{for } t \le T \\
0 & \text{otherwise}
\end{cases}.
\]

Let $G:\LT\rightarrow\LT$ be a linear operator. $G$ is said to be causal
if $P_T G P_T - P_T G = 0$ for all $T>0$. The induced norm of $G$ is defined to be
\[
\|G\|=\sup_{u\in\LT, u\not=0}\frac{\|Gu\|_{\LT}}{\|u\|_{\LT}}. 
\]
$G$ is said to be bounded if $\|G\| \le \gamma$ for some $\gamma > 0$. $G$ is said to be ``stable'' if $G$ is causal and bounded.  The adjoint of $G$
is denoted by $G^*$ and $G$ is said to be self-adjoint if $G=G^*$, in which case the notation $G\geq 0$ means $\langle u, G u \rangle_{\LT} \geq 0$
for all $u\in\LT$, and $G> 0$ means there exists $\epsilon >0$ such that $\langle u, G u \rangle_{\LT} \geq \epsilon\|u\|^2_{\LT}$ for all
$u\in\LT$. Finally, $G\le (<) 0$ means $-G\ge (>) 0$.

When $G$ commutes with the forward shift operator, it can be represented in the frequency domain as multiplication by a transfer function matrix,
which is denoted by $\hat{G}:\lambda\mapsto\hat{G}(\lambda)$. In this case, $G$ is called linear-time-invariant (LTI). It is well-known that when an
LTI system $G$ is stable, $\hat{G}$ is analytic and bounded in $\bS^c$, and 
\[
\|G\|= \|\hat{G}\|_\infty := \esssup_{\jw\in\pS}\bar{\sigma}(\hat{G}(\jw)). 
\]
The space of all such $\hat{G}$ is denoted by the symbol $\Hinf$. It is also well known that when $G$ is finite-dimensional LTI with a state-space
realization $(A,B,C,D)$, $\hat{G}(\jw)= C(\jw I-A)^{-1}B+D$, which belongs to the real rational subspace of $\Hinf$, denoted by $\RHinf$. When the
dimensions of $G$ are of significance, we write $\hat{G}\in\RHinf^{n\times m}$ to emphasize that $G$ has $m$ inputs and $n$ outputs. Let
$\cont$ denote the space of continuous functions on $\pS$. It is well known that any transfer function matrix $\hat{X} \in \Hinf \cap \cont$
satisfying $\hat{X}(\lambda^*)^T = \hat{X}(\lambda)^*$ can be approximated arbitrarily closely in $\|\cdot\|_\infty$ by elements in
$\RHinf$~\cite[Lemma A.6.11]{CurZwa95}.

A stable LTI system $G$ is called \emph{passive} if $\langle u,Gu\rangle_{\LT} \ge 0$ for any $u\in\LT$. 
It is called \emph{input strictly passive} if there exists $\epsilon>0$ such that 
$\langle u,Gu\rangle_{\LT} \ge \epsilon\|u\|_{\LT}^2$ for any $u\in\LT$, and 
\emph{output strictly passive} if there exists $\epsilon>0$ such that 
$\langle u,Gu\rangle_{\LT} \ge \epsilon\|Gu\|_{\LT}^2$ for any $u\in\LT$.
It is well known that $G$ is passive if and only if (iff) 
\[
\He(\hat{G}(\jw))\ge 0 \quad \text{for all} \quad \jw\in\pS, 
\]
$G$ is input strictly passive iff 
\[
\He(\hat{G}(\jw)) > 0 \quad \text{for all} \quad \jw\in\pS, 
\]
and $G$ is output strictly passive iff for some $\epsilon>0$,
\[
\He(\hat{G}(\jw)) \ge \epsilon \hat{G}(\jw)^*\hat{G}(\jw) \quad \text{for all} \quad \jw\in\pS. 
\]
Note that input strict passivity implies output strict passivity,
as 
\begin{align*}
 & \He(\hat{G}(\jw)) > 0 \implies \\
 & \exists \varepsilon>0, \mbox{ s.t. } \He(\hat{G}(\jw)) \ge \varepsilon I 
 \ge \frac{\varepsilon}{\|G\|^2} \hat{G}(\jw)^*\hat{G}(\jw).
\end{align*}
Thus, if we denote the sets of all (LTI) passive systems, output strictly passive systems, and input
strictly passive systems by $\mathcal{P}$, $\mathcal{P}_O$, and $\mathcal{P}_I$, respectively, we 
have the following relation
\begin{align}
\mathcal{P}_I \subset \mathcal{P}_O \subset \mathcal{P}. \label{Psets}
\end{align}
Notice that both inclusions are \emph{strict}. To see this, we note that the zero system is output strictly passive
but not input strictly passive; any non-zero skew symmetric matrix (viewed as a static system) is 
passive but not output strictly passive. 

The input passivity index of $G$ is the largest $\nu$ such that $\langle u,Gu\rangle_{\LT} \ge \nu\|u\|_{\LT}^2$ for any $u\in\LT$, or equivalently,
$\He(\hat{G}(\jw)) \geq 2\nu I$ for all $\jw\in\pS$. Evidently, $G$ is input strictly passive when $\nu > 0$. The output passivity
index of $G$ is the largest $\rho$ such that $\langle u,Gu\rangle_{\LT} \ge \rho\|Gu\|_{\LT}^2$ for any $u\in\LT$, or equivalently,
$\He(\hat{G}(\jw)) \ge 2\rho \hat{G}(\jw)^*\hat{G}(\jw)$ for all $\jw\in\pS$. Evidently, $G$ is output strictly passive when
$\rho > 0$. For more details on passivity indices, the reader is referred to~\cite{BaoLee07,KMXGA14}.

\section{Main Converse Results on IQCs} \label{sec: uncertainty}

\begin{figure}[ht]
  \centering
  \includegraphics[scale=0.55]{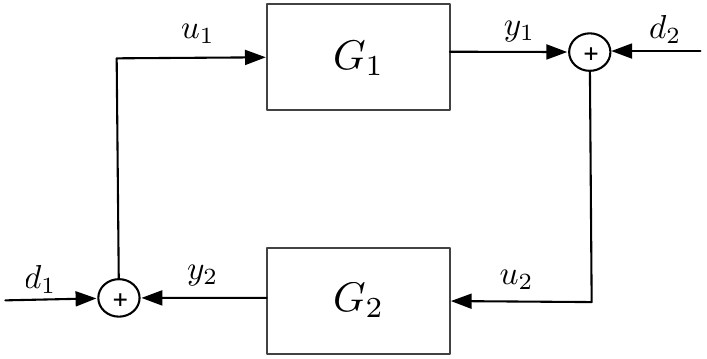}
  \caption{Positive feedback interconnection of $G_1$ and $G_2$.}
  \label{fig:1}
\end{figure}

Consider the
feedback interconnection of LTI causal systems $G_1$ and $G_2$ mapping $\LTE$ to $\LTE$, 
as illustrated in Figure~\ref{fig:1}. Algebraically, we have
\begin{align} \label{eq:FB}
\left\{\begin{array}{l}
u_2 = y_1 + d_2  \\
u_1 = y_2 + d_1 
\end{array}\right.
\quad
\left\{\begin{array}{l}
y_1 = G_1u_1 \\
y_2 = G_2u_2 
\end{array}.\right.
\end{align}
In the following, we denote the feedback interconnection of $G_1$ and $G_2$ by $[G_1, G_2]$.

\begin{definition} \label{def:FB} $[G_1, G_2]$ is said to be \emph{well-posed} if the map
  $(\STwoOne{u_1}{y_1}, \STwoOne{y_2}{u_2}) \mapsto \STwoOne{d_1}{d_2}$ defined by \eqref{eq:FB} 
  has a causal inverse on $\LTE$. It is \emph{stable} if it is well-posed and the inverse 
  is bounded.
\end{definition}

\begin{remark}
\label{rmk:1}
Note that when $[G_1,G_2]$ is well-posed, the map $\STwoOne{d_1}{d_2}\mapsto\STwoOne{y_1}{y_2}$ can 
be expressed as 
\begin{align*}
&\begin{bmatrix}G_1 & 0 \\ 0 & G_2\end{bmatrix} \begin{bmatrix}I & -G_2 \\ -G_1 & I \end{bmatrix}^{-1}
=\\
&\begin{bmatrix} G_1(I-G_2G_1)^{-1} & G_1(I-G_2G_1)^{-1}G_2 \\ (I-G_2G_1)^{-1}-I & (I-G_2G_1)^{-1}G_2 \end{bmatrix}
\end{align*}
Suppose $G_1$ and $G_2$ are both stable. Under this assumption, the above identity implies that
$[G_1,G_2]$ is stable if and only if $(I-\hat{G}_2\hat{G}_1)^{-1}\in\Hinf$. It can be shown that 
$(I-\hat{G}_2\hat{G}_1)^{-1}\in\Hinf$ if and only if $\det(I-\hat{G}_2(\jw)\hat{G}_1(\jw))\not=0$  for all 
$\jw\in\bS^c$. Moreover, by definition, stability of $[G_1,G_2]$ implies
that the maps $\STwoOne{d_1}{d_2}\mapsto\STwoOne{y_1}{y_2}$ and $\STwoOne{d_1}{d_2}\mapsto\STwoOne{u_1}{u_2}$
have finite gains. 
\end{remark}

Henceforth we also use $[G_1,G_2]$ to denote the map $\STwoOne{d_1}{d_2}\mapsto\STwoOne{y_1}{y_2}$. Suppose one of the systems, say $G_2$, is taken
from a set $\mathcal{U}$. We define uniform feedback stability in the following.

\begin{definition} \label{def:US} 
$[G_1,G_2]$ is said to be uniformly stable over $\mathcal{U}$ if $[G_1,G_2]$ is stable for all $G_2\in\mathcal{U}$, 
and there exists $\gamma > 0$ such that 
\[
\sup_{G_2\in\mathcal{U}} \left\|[G_1,G_2]\right\|\le \gamma.
\] 
\end{definition}

Let $\Pi$ be an $(n+m)\times (n+m)$, finite-dimensional LTI, bounded self-adjoint operator, and
partition $\Pi$ into $\TwoTwo{\Pi_{11}}{\Pi_{12}}{\Pi_{12}^*}{\Pi_{22}}$, such that the dimensions
of $\Pi_{11}$ and $\Pi_{22}$ are $n\times n$ and $m\times m$, respectively. Define the following $\Pi$-weighted 
quadratic forms:
\begin{align*}
&q_{\Pi}(G):=G^*\Pi_{11}G+G^*\Pi_{12}+\Pi_{12}^*G+\Pi_{22} \\ 
&q_{\Pi}^c(H):=\Pi_{11}+H^*\Pi_{12}^*+\Pi_{12}H+H^*\Pi_{22}H.
\end{align*}
Lastly, with $q_{\Pi}$ and $q_{\Pi}^c$, define the sets
\begin{align}
\label{IQC_sets}
\begin{split}
&\mathcal{G}_1:=\left\{G:\hat{G}\in\RHinf^{n \times m}, q_{\Pi}(G)<0 \right\} \quad \text{and} \\
&\mathcal{G}_2:=\left\{H:\hat{H}\in\RHinf^{m \times n}, q_{\Pi}^c(H)>0 \right\}
\end{split}
\end{align}
and let $\bar{\mathcal{G}}_1$ and $\bar{\mathcal{G}}_2$ be defined similarly to $\mathcal{G}_1$ and $\mathcal{G}_2$, but with ``$<$'' and ``$>$'' replaced
by ``$\le $'' and ``$\ge $'', respectively. Note that in the IQC literature~\cite{MegRan97, KLR16}, $q_{\Pi}(G)$ and $q_{\Pi}^c(H)$ are commonly
written in the following compact forms:
\begin{align*}
q_{\Pi}(G) = \TwoOne{G}{I}^*\Pi\TwoOne{G}{I} \text{ and } \; q_{\Pi}^c(H): = \TwoOne{I}{H}^*\Pi\TwoOne{I}{H}.
\end{align*}
Moreover, $\Pi$ is often referred to as a multiplier.

The main results of this paper are established in the following three theorems, which concern necessary and sufficient
conditions for (robust) stability and uniform stability. For the remainder of this section, we assume that
systems $G_1$ and $G_2$ are such that 
$\hat{G}_1\in\RHinf^{n\times m}$ and $\hat{G}_2\in\RHinf^{m\times n}$. 

\subsection{Robust Stability}

\begin{theorem}
\label{thm:main1}
Consider the feedback interconnected system shown in Figure~\ref{fig:1} with stable subsystems $G_1$ and $G_2$, 
the multiplier $\Pi$ and the sets $\mathcal{G}_i$, $\bar{\mathcal{G}}_i$, $i=1,2$, stated in (\ref{IQC_sets}). 
Suppose $\Pi$ admits a $J$-spectral factorization
\begin{align*}
\Pi = \Psi^*J\Psi:= \TwoTwo{\psi_1}{\psi_2}{\psi_3}{\psi_4}^*\TwoTwo{I}{0}{0}{-I}\TwoTwo{\psi_1}{\psi_2}{\psi_3}{\psi_4}
\end{align*}
such that the following conditions hold:
\begin{itemize}
\item[(1.1)] $\psi_i$, $i=1,\cdots,4$ are stable, and $\psi_4^{-1}$ is stable;  
\item[(1.2)] $\psi_1-\psi_2\psi_4^{-1}\psi_3$ is injective; 
\item[(1.3)] $\Pi_{11} := \psi_1^*\psi_1 - \psi_3^*\psi_3 \ge 0$;  
\item[(1.4)] $\Pi_{22} := \psi_2^*\psi_2 - \psi_4^*\psi_4 \le 0$. 
\end{itemize}
Then the system $[G_1,G_2]$ is stable for all $G_2\in\mathcal{G}_2$ if and only if $G_1\in\bar{\mathcal{G}}_1$. 
\end{theorem}

\begin{proof}
  Sufficiency follows the well-known quadratic separation theorem,~see~\cite{IwaHar98}, and also~\cite{MegRan97} where the roles of $G_1$ and $G_2$
  are swapped.  Here we only prove necessity. Suppose to the contrapositive that $G_1\not\in\bar{\mathcal{G}}_1$. Then
\begin{align}
\begin{split}
&(\psi_1G_1+\psi_2)^*(\psi_1G_1+\psi_2) \\ 
&\hspace{1.5cm} -(\psi_3G_1+\psi_4)^*(\psi_3G_1+\psi_4)\not\le 0. 
\end{split}\label{ineq:main1} 
\end{align}
If $\hat{\psi_3}(\jw)\hat{G}_1(\jw)+\hat{\psi}_4(\jw)$ is not invertible at $\jw = \tilde{\jw} \in \bS^c$, 
then let $G_2$ be $-\psi_4^{-1}\psi_3$.
We see that $I-\hat{G}_2\hat{G}_1$ is not invertible at $\tilde{\jw}$ and hence $[G_1,G_2]$ is not stable. 
On the other hand, 
\begin{align}
q_{\Pi}^c(G_2)=(\psi_1-\psi_2\psi_4^{-1}\psi_3)^*(\psi_1-\psi_2\psi_4^{-1}\psi_3),
\label{ineq:main2}
\end{align}
and thus $q_{\Pi}^c(G_2)>0$ since $\psi_1-\psi_2\psi_4^{-1}\psi_3$ is injective. This implies $G_2\in\mathcal{G}_2$. 

If $\hat{\psi}_3(\jw)\hat{G}_1(\jw)+\hat{\psi}_4(\jw)$ is invertible in $\bS^c$, then 
\[
M:=(\psi_1G_1+\psi_2)(\psi_3G_1+\psi_4)^{-1} 
\]
is stable, and inequality (\ref{ineq:main1}) implies that $\|M\|_{\LT}\not\le 1$.  Hence by the small gain theorem, there exists stable
finite-dimensional LTI operator $\Delta$ with $\|\Delta\|_{\LT} < 1$, such that $\det(I-\hat{\Delta}(\jw)\hat{M}(\jw))=0$ for some $\jw\in\pS$. For
the CT case, a constructive proof of this result can be found in~\cite[Theorem 9.1]{ZDG96}. For the DT case, the proof is similar.

Since $\psi_2^*\psi_2 - \psi_4^*\psi_4 \le 0$, we have $\|\psi_2\psi_4^{-1}\|_{\LT}\le 1$. Therefore, by the 
small gain theorem it follows that $I-\Delta\psi_2\psi_4^{-1}$ is invertible and 
$\psi_4^{-1}(I-\Delta\psi_2\psi_4^{-1})^{-1}=(\psi_4-\Delta\psi_2)^{-1}$ is stable. 
Define 
\[
G_2:=-\psi_4^{-1}(I-\Delta\psi_2\psi_4^{-1})^{-1}(\psi_3-\Delta\psi_1). 
\]
One can readily verify that $(I-\Delta M)(\psi_3G_1+\psi_4) = (\psi_4-\Delta\psi_2)(I-G_2G_1)$. Since $\psi_3G_1+\psi_4$ and $\psi_4-\Delta\psi_2$ are
both invertible, we see that $I-\hat{\Delta}(\jw)\hat{M}(\jw)$ being singular at some $\jw\in\pS$ implies the same for
$I-\hat{G}_2(\jw)\hat{G}_1(\jw)$, and hence $[G_1,G_2]$ is not stable. The relationship among $G_1$, $G_2$, $M$, and $\Delta$
can be interpreted by a chain-scattering formalism~\cite{Kim97}; see Figure~\ref{fig: LFT_feedback} for an illustration.

Finally, we show that $G_2:=-\psi_4^{-1}(I-\Delta\psi_2\psi_4^{-1})^{-1}(\psi_3-\Delta\psi_1)$ belongs to $\mathcal{G}_2$. To see 
this, let $\zeta_1 = \psi_1+\psi_2G_2$ and $\zeta_2 = \psi_3+\psi_4G_2$. We have the following equalities:
\begin{align*}
\zeta_2 &= \psi_3 - (I-\Delta\psi_2\psi_4^{-1})^{-1}(\psi_3-\Delta\psi_1)\\
&=-(I-\Delta\psi_2\psi_4^{-1})^{-1}\Delta\psi_2\psi_4^{-1}\psi_3 \\
&\hspace{3cm} + (I-\Delta\psi_2\psi_4^{-1})^{-1}\Delta\psi_1 \\
&=\Delta(I-\psi_2\psi_4^{-1}\Delta)^{-1}(\psi_1-\psi_2\psi_4^{-1}\psi_3) \\
\text{and} \\
\zeta_1 &= \psi_1-\psi_2\psi_4^{-1}(I-\Delta\psi_2\psi_4^{-1})^{-1}(\psi_3-\Delta\psi_1)\\
&=(I+\psi_2\psi_4^{-1}(I-\Delta\psi_2\psi_4^{-1})^{-1}\Delta)\psi_1\\ 
&\hspace{3cm} -\psi_2\psi_4^{-1}(I-\Delta\psi_2\psi_4^{-1})^{-1}\psi_3\\
&=(I-\psi_2\psi_4^{-1}\Delta)^{-1}(\psi_1-\psi_2\psi_4^{-1}\psi_3).
\end{align*}
Thus we have $\zeta_2 = \Delta\zeta_1$ and 
\[
q_{\Pi}^c(G_2)=\zeta_1^*(I-\Delta^*\Delta)\zeta_1>0,
\] 
as required.
\end{proof}

\begin{figure}[ht]
  \centering
  \includegraphics[scale=0.55]{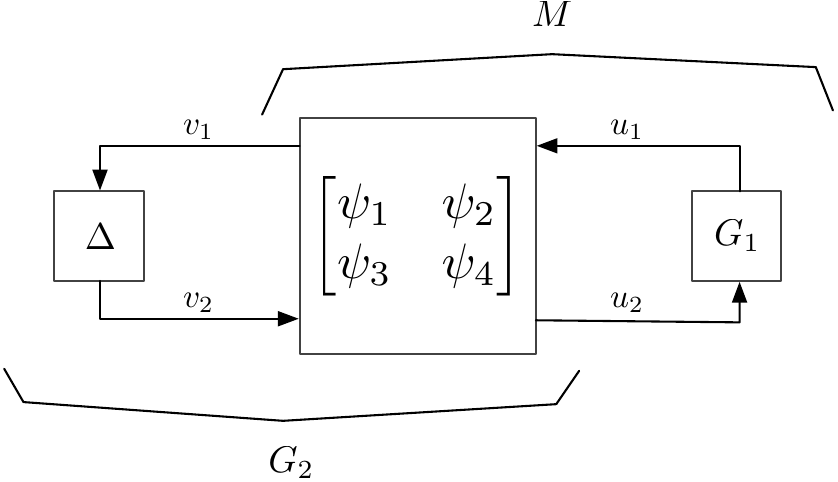}
  \caption{The chain-scattering transformation utilized in the proof of Theorem~\ref{thm:main1}. Note that the block in the middle 
    relates the signals $v_1$, $v_2$, $u_1$, and $u_2$ by the equations $v_1=\psi_1u_1+\psi_2u_2$ and $v_2=\psi_3u_1+\psi_4u_2$. 
    Thus, setting $v_1$ and $u_1$ to be $Mv_2$ and $G_1u_2$, respectively, the aforementioned equalities lead to an expression of $M$
    in terms of $G_1$ and $\psi_i$'s. Similarly, setting $v_2$ and $u_2$ to be $\Delta v_1$ and $G_2u_1$, we obtains an expression for $G_2$.}
  \label{fig: LFT_feedback}
\end{figure}

We can also prove the following theorem, where $\mathcal{G}_2$ is replaced by $\bar{\mathcal{G}}_2$ and $\bar{\mathcal{G}}_1$
by $\mathcal{G}_1$. Note that the conditions required for the $J$-spectral factors are slightly different in this case. 

\begin{theorem}
\label{thm:main2}
Consider the feedback interconnected system shown in Figure~\ref{fig:1} with stable subsystems $G_1$ and $G_2$, 
the multiplier $\Pi$ and the sets $\mathcal{G}_i$, $\bar{\mathcal{G}}_i$, $i=1,2$, stated in (\ref{IQC_sets}). 
Suppose $\Pi$ admits a $J$-spectral factorization
\begin{align*}
\Pi = \Psi^*J\Psi:= \TwoTwo{\psi_1}{\psi_2}{\psi_3}{\psi_4}^*\TwoTwo{I}{0}{0}{-I}\TwoTwo{\psi_1}{\psi_2}{\psi_3}{\psi_4}
\end{align*}
such that the following conditions hold:
\begin{itemize}
\item[(2.1)] $\psi_i$, $i=1,\cdots,4$ are stable, and $\psi_4^{-1}$ is stable;  
\item[(2.2)] $\Pi_{11} := \psi_1^*\psi_1 - \psi_3^*\psi_3 \ge 0$;  
\item[(2.3)] $\Pi_{22} := \psi_2^*\psi_2 - \psi_4^*\psi_4 < 0$. 
\end{itemize}
Then the system $[G_1,G_2]$ is stable for all $G_2\in\bar{\mathcal{G}}_2$ if and only if $G_1\in\mathcal{G}_1$. 
\end{theorem}

\begin{proof}
  Sufficiency follows from results in~\cite{MegRan97,IwaHar98}. Necessity can be proven by exactly the same arguments as those in the proof of
  Theorem~\ref{thm:main1}, except for the following minor differences. First, in the case where $\psi_3G_1+\psi_4$ is not boundedly invertible, we
  only need to show $G_2:=-\psi_4^{-1}\psi_3$ satisfies a non-strict quadratic inequality; i.e., to show that the quantity $q_{\Pi}^c(G_2)$ given
  in~\eqref{ineq:main2} is positive semi-definite. Hence the condition that $\psi_1-\psi_2\psi_4^{-1}\psi_3$ is injective is not required.

On the other hand, the $\Delta$ operator now satisfies $\|\Delta\|_{\LT}\le 1$ and hence we require 
$\Pi_{22}:=\psi_2^*\psi_2-\psi^*_4\psi_4 < 0$, in order to guarantee that 
$(\psi_4-\Delta\psi_2)^{-1}$ exists and is stable, which in turn enables the subsequent arguments. 
\end{proof}

\begin{remark}
  We note that the condition about the negative-definiteness of $\Pi_{22}$ is crucial (cf. (2.3) of Theorem~\ref{thm:main2}).  The condition allows
  the zero system to reside in the set $\mathcal{G}_1$. If this were not the case, the necessity part of the theorem would be invalid, as $[0,G_2]$ is stable
  for any stable $G_2$ (and all systems in $\bar{\mathcal{G}}_2$ are stable).
\end{remark}

\subsection{Robust Uniform Stability}

One can relax the strict negativeness on $\Pi_{22}$ by enforcing uniformity on closed-loop stability. Specifically, we
 have the following theorem. 

\begin{theorem}
\label{thm:main3}
Consider the feedback interconnected system shown in Figure~\ref{fig:1} with stable subsystems $G_1$ and $G_2$, 
the multiplier $\Pi$ and the sets $\mathcal{G}_i$, $\bar{\mathcal{G}}_i$, $i=1,2$, stated in (\ref{IQC_sets}). 
Suppose $\Pi$ admits a $J$-spectral factorization
\begin{align*}
\Pi = \Psi^*J\Psi:= \TwoTwo{\psi_1}{\psi_2}{\psi_3}{\psi_4}^*\TwoTwo{I}{0}{0}{-I}\TwoTwo{\psi_1}{\psi_2}{\psi_3}{\psi_4}
\end{align*}
such that the following conditions hold:
\begin{itemize}
\item[(3.1)] $\psi_i$, $i=1,\cdots,4$ are stable, and $\psi_4^{-1}$ is stable;  
\item[(3.2)] $\Pi_{11} := \psi_1^*\psi_1 - \psi_3^*\psi_3 \ge 0$;  
\item[(3.3)] $\Pi_{22} := \psi_2^*\psi_2 - \psi_4^*\psi_4 \le 0$. 
\end{itemize}
Then the system $[G_1,G_2]$ is \emph{uniformly} stable for any $G_2\in\bar{\mathcal{G}}_2$ if and only if $G_1\in\mathcal{G}_1$. 
\end{theorem}

\begin{proof}
If $G_1\in\mathcal{G}_1$, then stability of $[G_1,G_2]$ for any $G_2\in\bar{\mathcal{G}}_2$ is proven in~\cite{IwaHar98} and~\cite{MegRan97}.
Uniform stability of $[G_1,G_2]$ is inferred from the proof of stability in~\cite{MegRan97}, where the bound on the gain is shown. 
See also Lemma 4.1.2 of~\cite{Kho11} for a proof of uniform stability on a more general setting, for which the LTI setting considered here
is a special case. 

For necessity, the arguments follow similar lines as presented in the proof of Theorem~\ref{thm:main2}.  Suppose $G_1\not\in\mathcal{G}_1$ and
$\psi_3G_1+\psi_4$ is not boundedly invertible. Then by selecting $G_2$ as $-\psi_4^{-1}\psi_3\in\bar{\mathcal{G}}_2$, we have
$G_2\in\bar{\mathcal{G}}_2$ and $[G_1,G_2]$ is unstable.

On the other hand, if $\hat{\psi}_3(\jw)\hat{G}_1(\jw)+\hat{\psi}_4(\jw)$ is invertible in $\bS^c$, then
\[
M:=(\psi_1G_1+\psi_2)(\psi_3G_1+\psi_4)^{-1} 
\]
is stable, and $G_1\not\in\mathcal{G}_1$ implies that $\|M\|_{\LT}\not < 1$.  Hence by the small gain theorem, 
there exists stable LTI operator $\Delta$ with $\|\Delta\|_{\LT} \le 1$, such that $\det(I-\hat{\Delta}(\jw)\hat{M}(\jw))=0$ 
for some $\jw\in\pS$, which implies that $[M,\Delta]$ is not stable. We will now show that, by $\Delta$, it 
is possible to either destabilize $[G_1,G_2]$ by some $G_2\in\bar{\mathcal{G}}_2$, or construct a
series of $G_2$'s in $\bar{\mathcal{G}}_2$ such that the gain of $[G_1,G_2]$ becomes arbitrarily large.

Since $\psi_2^*\psi_2 - \psi_4^*\psi_4 \le 0$, we have $\|\psi_2\psi_4^{-1}\|_{\LT}\le 1$. As $\|\Delta\|_{\LT} \le 1$, 
$(I-\Delta \psi_2\psi_4^{-1})$ may not be invertible. Thus we let $\Delta_\rho:=\rho\cdot\Delta$ with $0 < \rho<1$
and 
\begin{align}
\label{G2r}
G_{2\rho}:=-\psi_4^{-1}(I-\Delta_\rho\psi_2\psi_4^{-1})^{-1}(\psi_3-\Delta_\rho\psi_1). 
\end{align}
Clearly, $G_{2\rho}$ is well-defined and stable for all $\rho < 1$, as $\|\Delta_\rho\|_{\LT} < 1$.  Now one can readily verify that
\begin{align*}
&(I-G_{2\rho}G_1)^{-1}G_{2\rho} \\
&\hspace{0.5cm} = -(\psi_3G_1+\psi_4)^{-1}(I-\Delta_\rho M)^{-1}(\psi_3-\Delta_\rho\psi_1).
\end{align*}
Hence $\|(I-G_{2\rho}G_1)^{-1}G_{2\rho}\|\rightarrow\infty$ as $\rho\rightarrow 1$, because 
$\det(I-\hat{\Delta}_\rho(\jw)\hat{M}(\jw))\rightarrow 0$ for some $\jw \in \pS$ as $\rho\rightarrow 1$. This in turn
implies $\|[G_1,G_{2\rho}]\|\rightarrow\infty$ as $\rho\rightarrow 1$. Lastly, to see that $G_{2\rho}\in\bar{\mathcal{G}}_2$,
we note that the derivation in the last paragraph of the proof of Theorem~\ref{thm:main1} remains entirely the same when
$\Delta$ is replaced by $\Delta_\rho$. Hence $q_{\Pi}^c(G_{2\rho})=\zeta_1^*(I-\rho^2\Delta^*\Delta)\zeta_1>0$
for all $\rho<1$. Thus, we conclude that $[G_1,G_2]$ is not uniformly stable over $\bar{\mathcal{G}}_2$. 
\end{proof}

\section{Converse Passivity Theorems} \label{sec: passive}

In this section, we apply Theorems~\ref{thm:main1} to~\ref{thm:main3} to derive multiple converse passivity theorems.  First, we note that, a minus
sign will be applied to system $G_1$ as in $[-G_1, G_2]$ throughout this section in order to stay in line with the negative feedback convention in the passivity
literature. In Section~\ref{sec:stab}, converse passivity theorems based on the notion of robust closed-loop stability are discussed, while theorems
related to the notion of robust \emph{uniform} stability are discussed in~\ref{sec:unistab}.

\subsection{Robust Stability}
\label{sec:stab}

The following proposition follows from Theorem~\ref{thm:main1} by taking an appropriate multiplier $\Pi$ and $\Pi$-weighted quadratic forms. 

\begin{proposition}
\label{prop1}
Consider the feedback interconnected system $[-G_1,G_2]$ such as the one shown in Figure~\ref{fig:1}, where the subsystems $G_1$ and $G_2$ are square and stable.
Then the system is stable for any input strictly passive $G_2$ if and only if $G_1$ is passive. 
\end{proposition}

\begin{proof}
Let $\Pi = \begin{bmatrix}0 & I \\ I & 0\end{bmatrix}$. It can be 
readily verified that $\Pi$ has the required $J$-spectral factorization with $\psi_1 = \psi_2 = \psi_3 = \frac{1}{\sqrt{2}}I$ 
and $\psi_4 = -\frac{1}{\sqrt{2}}I$. 
Clearly, $\psi_4^{-1}$ is stable and $\psi_1-\psi_2\psi_4^{-1}\psi_3 = \sqrt{2}I$ is injective. Thus, the conditions
required for $J$-spectral factors are satisfied. By definition, the set $\mathcal{G}_2$ with quadratic form $q_{\Pi}^c(\cdot)$ 
is the set of all (LTI) input strictly passive systems. To see this, note that
\begin{align*}
G_2\in\mathcal{G}_2  \Leftrightarrow 
\begin{array}{l}
 \exists~\epsilon>0 \mbox{ s.t. } \left\langle u, (G_2^*+G_2)u \right\rangle_{\LT} = \\ 
 2\left\langle u, G_2 u \right\rangle_{\LT} \ge \epsilon\|u\|_{\LT}^2, \forall u\in\LT.
\end{array}
\end{align*}
Likewise, one can readily verify that
\begin{align*}
-G_1\in\bar{\mathcal{G}}_1   \Leftrightarrow 
\begin{array}{l} 
 2\left\langle u, -G_1 u \right\rangle_{\LT} 
\equiv -2\left\langle u, G_1 u \right\rangle_{\LT} 
\le 0, \\ 
\hspace{4.25cm} \forall u\in\LT.
\end{array}
\end{align*}
Hence $G_1$ is passive. This concludes the proof. 
\end{proof}

\begin{remark}
\label{rmk:aftProp1}
Since the sets $\mathcal{P}_I$, $\mathcal{P}_O$, and $\mathcal{P}$ satisfy the strict inclusion relationship described in (\ref{Psets}), we immediately 
have the following conclusions by Proposition~\ref{prop1}:
\begin{itemize}
\item $G_1$ being passive is \emph{necessary} for $[-G_1,G_2]$ to be stable for 
all output strictly passive $G_2$, and in fact, for all passive $G_2$. 
\item $G_1$ being output strictly passive is \emph{sufficient but not necessary} for $[-G_1,G_2]$ to be stable for all input strictly
passive $G_2$. 
\item $G_1$ being input strictly passive is \emph{sufficient but not necessary} for $[-G_1,G_2]$ to be stable for all input strictly
passive $G_2$. 
\end{itemize}
To see the ``not necessary'' part of the last two statements, take any non-zero skew-symmetric matrix as $G_1$, which is passive 
but not output strictly nor input strictly passive. The sufficiency direction stated in Proposition~\ref{prop1} yields that such $G_1$ will result in stable $[G_1,
G_2]$ for all input strictly passive $G_2$.
\end{remark}

Furthermore, the condition that $G_1$ is passive can also be proven to be sufficient for $[-G_1,G_2]$ to be stable over $\mathcal{P}_O$. 
This leads to the following necessary and sufficient condition. 

\begin{proposition}
\label{prop2}
Consider the feedback interconnected system $[-G_1,G_2]$ such as the one shown in Figure~\ref{fig:1}, where 
the subsystems $G_1$ and $G_2$ are square and stable.
Then the system is stable for any output strictly passive $G_2$ if and only if $G_1$ is passive. 
\end{proposition}

\begin{proof}
  Necessity is established in Proposition~\ref{prop1} and Remark~\ref{rmk:aftProp1}. Sufficiency is in fact well-known, see e.g.~\cite{BaoLee07,
    GreLim95, Sch17}.  Here we show that the result can also be obtained by applying Theorem~\ref{thm:main1}.  Let $\epsilon>0$ be any positive real
  number and $\Pi = \begin{bmatrix}0 & I \\ I & -\epsilon I \end{bmatrix}$, which has the required $J$-spectral factorization with
  $\psi_1 = \psi_3 = (\sqrt{\epsilon}-\sqrt{2\epsilon})^{-1}I$, $\psi_2 = \sqrt{\epsilon}I$, and
  $\psi_4=\sqrt{2\epsilon}I$. 
  Thus, by applying Theorem~\ref{thm:main1} we obtain the following condition: $[-G_1,G_2]$ is stable for all $G_2$ satisfying
  $\He(\hat{G}_2(\jw)) > \epsilon \hat{G}_2(\jw)^*\hat{G}_2(\jw)$ for all $\jw\in\partial\bS$ if and only if
\begin{align}
\He(\hat{G}_1(\jw)) \ge -\epsilon I \mbox{ for all } \jw\in\partial\bS. 
\label{ineq:prop2}
\end{align}
If $G_1$ is passive, then \eqref{ineq:prop2} holds for any $\epsilon >0$. This in turn 
implies $[G_1,G_2]$ is stable for each and every $G_2$ that is output strictly passive. 
\end{proof}

Lastly, the following sufficient condition can also be proven by applying Theorem~\ref{thm:main1}.

\begin{proposition}
\label{prop3}
Consider the feedback interconnected system $[-G_1,G_2]$ such as the one shown in Figure~\ref{fig:1}, where 
the subsystems $G_1$ and $G_2$ are square and stable.
If $G_1$ is output strictly passive, then the system is stable for any passive $G_2$. 
\end{proposition}

\begin{proof}
Let $\epsilon>0$ be such that $\He(\hat{G}_1(\jw)) \ge\epsilon\hat{G}_1(\jw)^*\hat{G}_1(\jw)$ for all $\jw\in\partial\bS$.
Define $\Pi:=\begin{bmatrix}\epsilon I & I \\ I & 0 \end{bmatrix}$ and one can readily verify that $G_1\in\bar{\mathcal{G}}_1$,  
where $\bar{\mathcal{G}}_1$ defined by the quadratic form $q_\Pi(\cdot)$. Moreover, $\Pi$ has the required $J$-spectral factorization with
$\psi_1=\sqrt{2\epsilon}I$, $\psi_3=\sqrt{\epsilon}I$, $\psi_2=\psi_4=(\sqrt{2\epsilon}-\sqrt{\epsilon})^{-1}I$. 
Thus, by applying Theorem~\ref{thm:main1}, we conclude
that $[G_1,G_2]$ is stable for all $G_2$ satisfying $\He(\hat{G}_2(\jw)) > -\epsilon I$ for all $\jw\in\partial\bS$, which
in turn implies $[G_1,G_2]$ is stable for all passive $G_2$. 
\end{proof}

\begin{remark}
\label{rmk:aftProp3}
As before, the strict inclusion relationship described in (\ref{Psets}) together with the sufficient conditions stated in 
Propositions~\ref{prop2} and~\ref{prop3} immediately lead to the following conclusions
\begin{itemize}
\item $G_1$ being output strictly passive is \emph{sufficient but not necessary} for $[-G_1,G_2]$ to be stable for all output strictly
passive $G_2$. 
\item $G_1$ being input strictly passive is \emph{sufficient but not necessary} for $[-G_1,G_2]$ to be stable for all output strictly
passive $G_2$. 
\item $G_1$ being input strictly passive is \emph{sufficient but not necessary} for $[-G_1,G_2]$ to be stable for all 
passive $G_2$. 
\end{itemize}
\end{remark}

\noindent 
Finally, we note that it is well-known that $G_1$ being passive is not sufficient for $[-G_1,G_2]$ to be stable for all
passive $G_2$. Take $G_1 = G_2 = \begin{bmatrix} 0 & I \\ -I & 0 \end{bmatrix}$ for example; 
$G_1$ and $G_2$ are both passive but $I+G_1G_2=0$, which is not invertible. As such, $[-G_1,G_2]$ is not even 
well-posed, let alone stable.  
Table~\ref{tab:stab} summarizes the conditions for \emph{robust stability} we have discovered so far. 
\begin{table}
\label{tab:stab}
\begin{center}
Conditions for robust stability of $[-G_1,G_2]$ over a set of $G_2$.

\vspace{0.25cm}
\begin{tabular}{|l||c|c|c|}
  \hline
  & $G_1\in\mathcal{P}$   & $G_1\in\mathcal{P}_O$          & $G_1\in\mathcal{P}_I$  \\
  \hhline{|=||=|=|=|}
  $\forall G_2\in\mathcal{P}$   & $\mathrm{N}\cancel{\mathrm{S}}$           & $\mathrm{S}$                    & $\cancel{\mathrm{N}}\mathrm{S}$   \\
  \hline
  $\forall G_2\in\mathcal{P}_O$ & $\mathrm{N}\mathrm{S}$ & $\cancel{\mathrm{N}}\mathrm{S}$ & $\cancel{\mathrm{N}}\mathrm{S}$     \\
  \hline
  $\forall G_2\in\mathcal{P}_I$ & $\mathrm{N}\mathrm{S}$ & $\cancel{\mathrm{N}}\mathrm{S}$ & $\cancel{\mathrm{N}}\mathrm{S}$  \\
  \hline
\end{tabular}
\end{center}
\caption{}
\vspace{-0.25cm}
{\footnotesize
N~/~$\cancel{\mathrm{N}}$: the condition in the top row (is / is not) necessary for stability over the set in the first column.\\
S~/~$\cancel{\mathrm{S}}$: the condition in the top row (is / is not) sufficient for stability over the set in the first column.\\
}
\end{table}

\subsection{Robust Uniform Stability}
\label{sec:unistab}

If we impose uniformity on closed-loop stability, then Theorem~\ref{thm:main3} can be applied to obtain the following necessary and sufficient 
condition for robustness against passivity, which can be viewed as the dual of Proposition~\ref{prop1}. 

\begin{proposition}
\label{prop4}
Consider the feedback interconnected system $[-G_1,G_2]$ such as the one shown in Figure~\ref{fig:1}, where the subsystems 
$G_1$ and $G_2$ are square and stable.
Then the system is uniformly stable for any passive $G_2$ if and only if $G_1$ is input strictly passive. 
\end{proposition}

\begin{proof}
Let $\Pi = \begin{bmatrix}0 & I \\ I & 0\end{bmatrix}$. It has been established in Proposition~\ref{prop1} that 
the $\Pi$ has the $J$-spectral factorization which satisfies conditions (1.1) to (1.4), and therefore also 
conditions (3.1) to (3.3). Furthermore, by the arguments similar to those in the proof of Proposition~\ref{prop1}, one
can readily verify that with this $\Pi$, $\bar{\mathcal{G}}_2$ is the set of all LTI passive systems, while 
\begin{align*}
-G_1\in\mathcal{G}_1 \Leftrightarrow 
\begin{array}{l}
\exists~\epsilon>0 \mbox{ s.t. } \left\langle u, -(G_1^*+G_1)u \right\rangle_{\LT}= \\
-2\left\langle u, G_1 u \right\rangle_{\LT} \le -\epsilon\|u\|_{\LT}^2, \forall u\in\LT.
\end{array}
\end{align*}
Hence $G_1$ is input strictly passive. This concludes the proof. 
\end{proof}

\begin{remark}
\label{rmk:aftProp4}
By Proposition~\ref{prop4} and the strict inclusion relationship~\eqref{Psets}, we may arrive at the following conclusions:
\begin{itemize}
\item $G_1$ being input strictly passive is \emph{sufficient} for $[-G_1,G_2]$ to be uniformly stable for 
all output strictly passive $G_2$ and for all input strictly passive $G_2$. 
\item $G_1$ being output strictly passive is \emph{necessary} for $[-G_1,G_2]$ to be uniformly stable for all passive $G_2$. 
\item $G_1$ being passive is \emph{necessary} for $[-G_1,G_2]$ to be uniformly stable for all passive $G_2$.
\end{itemize}
\end{remark}

The following necessary and sufficient conditions also follow immediately from Theorem~\ref{thm:main3} by taking appropriate $\Pi$'s. 

\begin{proposition}
\label{prop5}
Consider the feedback interconnected system $[-G_1,G_2]$ such as the one shown in Figure~\ref{fig:1}, 
where the subsystems $G_1$ and $G_2$ are square and stable.
\begin{itemize}
\item[(\ref{prop5}.1)] Given any $\varepsilon>0$, the system is uniformly stable for all $G_2$ satisfying 
$\hat{G}_2(\jw)^*+\hat{G}_2(\jw)\ge\varepsilon \hat{G}_2(\jw)^*\hat{G}_2(\jw)$~$\forall\jw\in\partial\bS$
if and only if $G_1$ satisfies $\hat{G}_1(\jw)^*+\hat{G}_1(\jw)>-\varepsilon I$~$\forall\jw\in\partial\bS$. 
\item[(\ref{prop5}.2)] Given any $\varepsilon>0$, the system is uniformly stable for all $G_2$ satisfying 
$\hat{G}_2(\jw)^*+\hat{G}_2(\jw)\ge-\varepsilon I$~$\forall\jw\in\partial\bS$
if and only if $G_1$ satisfies $\hat{G}_1(\jw)^*+\hat{G}_1(\jw)> \varepsilon \hat{G}_1(\jw)^*\hat{G}_1(\jw)$~$\forall\jw\in\partial\bS$. 
\end{itemize}
\end{proposition}
\begin{proof}
Statements (\ref{prop5}.1)~and~(\ref{prop5}.2) are obtained by applying Theorem~\ref{thm:main3} with 
$\Pi = \begin{bmatrix}0 & I \\ I & -\varepsilon I\end{bmatrix}$ and $\Pi = \begin{bmatrix}\varepsilon I & I \\ I & 0\end{bmatrix}$, respectively. 
The arguments are similar to those in the previous propositions. Here we omit the details. 
\end{proof}

By combining Propositions~\ref{prop4} and~\ref{prop5}, we obtain the following result that relates the passivity indices of $G_1$ to those of $G_2$.

\begin{proposition}
\label{prop:indices}
Consider the feedback interconnected system $[-G_1,G_2]$ such as the one shown in Figure~\ref{fig:1}, where the subsystems $G_1$ and $G_2$
are square and stable. Let the
input and output passivity indices of $G_1$ be $\nu_1$ and $\rho_1>0$, respectively. Then, given $\nu_2 \in \Real$ and $\rho_2 \ge 0$,
\begin{itemize}
\item[(\ref{prop:indices}.1)] the system is uniformly stable for all $G_2$ with output passivity index at least $\rho_2$ if and only if
  $\nu_1 + \rho_2 > 0$.
\item[(\ref{prop:indices}.2)] the system is uniformly stable for all $G_2$ with input passivity index at least $\nu_2$ if and only if
  $\rho_1 + \nu_2 > 0$.
\end{itemize}
\end{proposition}

The first statement in the Proposition~\ref{prop:indices} provides a lower bound on the input passivity deficit in $G_1$ for which an excess of output
passivity in an arbitrary $G_2$ can compensate. The second statement lower bounds the output passivity surplus in $G_1$ that is needed to compensate
for a lack of input passivity in an arbitrary $G_2$. Sufficiency of these statements is well known in the literature; see~\cite{BaoLee07, KMXGA14} and
the references therein. The proofs of necessity given in this paper are novel.

Statement (\ref{prop5}.1) leads to the following conditions regarding uniform stability over the set of output strictly passive systems. 

\begin{proposition}
\label{prop6}
Consider the feedback interconnected system $[-G_1,G_2]$ such as the one shown in Figure~\ref{fig:1}, where 
the subsystems $G_1$ and $G_2$ are square and stable.
\begin{itemize}
\item[(\ref{prop6}.1)]
If the system is uniformly stable for all output strictly passive $G_2$, then $G_1$ is passive. 
\item[(\ref{prop6}.2)] If $G_1$ is passive, then for any $\varepsilon>0$, the system is uniformly stable for all $G_2$ with output passivity index $\varepsilon$.
\end{itemize}
\end{proposition}
\begin{proof}
Statement (\ref{prop6}.2) follows straightforwardly the sufficiency part of statement (\ref{prop5}.1). To establish statement
(\ref{prop6}.1), note that if $G_1$ is not passive, then there exists $\varepsilon>0$ such that 
$\He(\hat{G}_1(\jw))\not>-\varepsilon I$ at some $\jw\in\partial\bS$. Hence by the necessity part of statement (\ref{prop5}.1),
$G_1$ would fail to make $[-G_1,G_2]$ uniformly stable for all $G_2$ satisfying 
$\He(\hat{G}_2(\jw))\ge\varepsilon \hat{G}_2(\jw)^*\hat{G}_2(\jw)$~$\forall\jw\in\partial\bS$, and thus
$[-G_1,G_2]$ is not uniformly stable for all output strictly passive $G_2$. 
\end{proof}

\begin{remark}
\label{rmkAftprop6}
Another straightforward argument for establishing statement (\ref{prop6}.1) is to note that $G_1$ being passive is necessary
for stability of $[-G_1,G_2]$ to hold over all output strictly passive $G_2$. Hence the same must hold for uniform stability, 
since the latter is a stronger notion than the former. By the same token, $G_1$ being passive is also necessary for uniform
stability of $[-G_1,G_2]$ to hold over all input strictly passive $G_2$.
\end{remark}

One may notice that there is a gap between the necessary condition (\ref{prop6}.1) and the sufficient condition (\ref{prop6}.2). 
The following proposition shows that this gap cannot be closed. 

\begin{proposition}
\label{prop7}
$G_1$ being output strictly passive is not sufficient for uniform stability of $[-G_1,G_2]$ to hold over all output strictly passive 
$G_2$, nor in fact, over all input strictly passive $G_2$. 
\end{proposition}

\begin{proof}
To see this, one simply needs to note that the zero system is output strictly passive, and the sets of all
output strictly passive systems $\mathcal{P}_O$ and input strictly passive systems $\mathcal{P}_I$ both 
contain systems whose gains are arbitrarily large. Therefore $[0, G_2]$ can never be uniformly stable over
$\mathcal{P}_O$ or $\mathcal{P}_I$.
\end{proof}

As the set of output strictly passive systems is contained in the 
set of passive systems, it is clear from Proposition~\ref{prop7} that $G_1$ being passive 
is also not sufficient for uniform stability of $[-G_1,G_2]$ to hold over $\mathcal{P}_O$ or $\mathcal{P}_I$.

Table~\ref{tab2} summarizes the conditions for \emph{robust uniform stability} we have discovered so far. 
\begin{table}
Conditions for robust uniform stability of $[-G_1,G_2]$ over a set of $G_2$.

\vspace{0.1cm}
\label{tab2}
\begin{center}
\begin{tabular}{|l||c|c|c|}
\hline
                              & $G_1\in\mathcal{P}$            & $G_1\in\mathcal{P}_O$           & $G_1\in\mathcal{P}_I$  \\
\hhline{|=||=|=|=|}
$\forall G_2\in\mathcal{P}$   & $\mathrm{N}\cancel{\mathrm{S}}$ & $\mathrm{N}\cancel{\mathrm{S}}$  & $\mathrm{N}\mathrm{S}$   \\
\hline
$\forall G_2\in\mathcal{P}_O$ & $\mathrm{N}\cancel{\mathrm{S}}$ & $\cancel{\mathrm{S}}$            & $\mathrm{S}$     \\
\hline
$\forall G_2\in\mathcal{P}_I$ & $\mathrm{N}\cancel{\mathrm{S}}$           & $\cancel{\mathrm{S}}$            & $\mathrm{S}$  \\
\hline
\end{tabular}
\end{center}
\caption{}
\vspace{-0.25cm}
{\footnotesize
N: the condition in the top row is necessary for stability over the set in the first column.\\
S~$\cancel{\mathrm{S}}$: the condition in the top row (is / is not) sufficient for stability over the set in the first column.
}
\end{table}

\section{Generalizations to infinite-dimensional multipliers} \label{sec: inf_dim}

In this section, we derive a generalization of Theorem~\ref{thm:main3} to the case where the multiplier involved is not restricted to be of finite
dimension. This is then specialized to deriving a couple of interesting results, namely frequency-weighted small-gain and passivity theorems.

Let $\Pi$ be an $(n+m)\times (n+m)$ LTI and bounded self-adjoint operator. Note that unlike Section~\ref{sec: uncertainty}, the multiplier $\Pi$ is
not required to be finite-dimensional here. The following result is in order.

\begin{theorem}
\label{thm:inf_dim}
Consider the feedback interconnected system shown in Figure~\ref{fig:1} with stable subsystems $G_1$ and $G_2$, 
the multiplier $\Pi$ and the sets $\mathcal{G}_i$, $\bar{\mathcal{G}}_i$, $i=1,2$, stated in (\ref{IQC_sets}). 
Suppose $\Pi$ admits a $J$-spectral factorization
\begin{align*}
\Pi = \Psi^*J\Psi:= \TwoTwo{\psi_1}{\psi_2}{\psi_3}{\psi_4}^*\TwoTwo{I}{0}{0}{-I}\TwoTwo{\psi_1}{\psi_2}{\psi_3}{\psi_4}
\end{align*}
such that the following conditions hold:
\begin{itemize}
\item[(\ref{thm:inf_dim}.1)] $\hat{\Psi}, \hat{\psi}_4^{-1} \in \Hinf \cap \cont$ and $\hat{\Psi}(\lambda^*)^T = \hat{\Psi}(\lambda)^*$ for all
  $\lambda \in \pS$;
\item[(\ref{thm:inf_dim}.2)] $\Pi_{11} := \psi_1^*\psi_1 - \psi_3^*\psi_3 \ge 0$;  
\item[(\ref{thm:inf_dim}.3)] $\Pi_{22} := \psi_2^*\psi_2 - \psi_4^*\psi_4 \le 0$. 
\end{itemize}
Then the system $[G_1,G_2]$ is \emph{uniformly} stable for any $G_2\in\bar{\mathcal{G}}_2$ if and only if $G_1\in\mathcal{G}_1$. 
\end{theorem}

\begin{proof}
  The proof is largely similar to that of Theorem~\ref{thm:main3}, with the exception that in the necessity direction one would 
  need to employ the argument that any transfer function matrix $\hat{X} \in \Hinf \cap \cont$ satisfying 
  $\hat{X}(\lambda^*)^T = \hat{X}(\lambda)^*$ can be approximated arbitrarily closely in $\|\cdot\|_\infty$ by elements in $\RHinf$.

  More specifically, note that $G_{2\rho}$ as defined in~(\ref{G2r}), though satisfies $q_{\Pi}^c(G_{2\rho})>0$ ($\forall~0<\rho<1$), does not
  necessarily belong to $\bar{\mathcal{G}}_2$ because $\psi_i$, $i=1,\cdots,4$, may not be rational. To complete the remaining steps of the proof, let
  $G_{2\rho}^r$ belong to $\RHinf$ and note that
  \begin{align*}
  &(I-G_{2\rho}^rG_1)^{-1}G_{2\rho}^r-(I-G_{2\rho}G_1)^{-1}G_{2\rho}\\
  &=(I-G_{2\rho}^rG_1)^{-1}(G_{2\rho}^r-G_{2\rho}) + \\
  & \ \ (I-G_{2\rho}^rG_1)^{-1}(G_{2\rho}^r-G_{2\rho})G_1(I-G_{2\rho}G_1)^{-1}G_{2\rho}.
  \end{align*}
  Thus, 
  \begin{align*}
  \|(I-G_{2\rho}^rG_1)^{-1}G_{2\rho}^r-(I-G_{2\rho}G_1)^{-1}G_{2\rho}\| \le \\ 
  \|G_{2\rho}^r-G_{2\rho}\|\left(n_1(\rho)+n_1(\rho)n_2(\rho)\right),
  \end{align*}
  where $n_1(\rho):=\|(I-G_{2\rho}^rG_1)^{-1}\|$ and $n_2(\rho):=\|G_1(I-G_{2\rho}G_1)^{-1}G_{2\rho}\|$. 
  The above inequality implies that, given any $\rho\in(0,1)$ and $G_{2\rho}$ defined in~(\ref{G2r}), one can find 
  $G_{2\rho}^r\in\RHinf$ such that $\|G_{2\rho}^r-G_{2\rho}\|$ is sufficiently small and
  \begin{align*}
  \|(I-G_{2\rho}^rG_1)^{-1}G_{2\rho}^r\|\ge \|(I-G_{2\rho}G_1)^{-1}G_{2\rho}\|-c,
  \end{align*}
  where $c$ is a constant that upper bounds 
\[
\|G_{2\rho}^r-G_{2\rho}\|\left(n_1(\rho)+n_1(\rho)n_2(\rho)\right).
\] 
Note that such a constant exists because $\|G_{2\rho}^r-G_{2\rho}\|=\|\hat{G}_{2\rho}^r-\hat{G}_{2\rho}\|_{\infty}$ can be made arbitrarily small for
any given $G_{2\rho}$.  Finally, notice that since $q_{\Pi}^{c}(\cdot)$ is quadratic and hence continuous, $q_{\Pi}^{c}(G_{2\rho})>0$ implies 
$q_{\Pi}^{c}(G^r_{2\rho}) > 0$ whenever $\|G_{2\rho}^r-G_{2\rho}\|$ is sufficiently small. Thus, we have shown that for every $G_{2\rho}$ constructed in
Theorem~\ref{thm:main3}, one can find a $G_{2\rho}^r\in\bar{\mathcal{G}}_2$ such that $\|[G_1,G_{2\rho}^r]\|$ is lower bounded by
$\|(I-G_{2\rho}G_1)^{-1}G_{2\rho}\|-c$, which diverges to infinity as $\rho\rightarrow 1$.
\end{proof}

\subsection{Frequency-weighted small-gain theorem}

The following result is a generalization of the well-known small-gain theorem.

\begin{proposition}
  Let $\hat{\psi}_1 \in \Hinf \cap \cont$ be a scalar function satisfying $\hat{\psi}_1(\lambda^*) = \hat{\psi}_1(\lambda)^*$ for all
  $\lambda \in \pS$, and $\hat{\gamma}(\lambda) := |\hat{\psi}_1(\lambda)|^2$. Then given $\hat{G}_1 \in \RHinf$, $[G_1, G_2]$ is uniformly stable
  over all $\hat{G}_2 \in \RHinf$ satisfying
\[
\bar{\sigma}(\hat{G}_2(\lambda)) \leq \hat{\gamma}(\lambda) \quad \forall \lambda \in \pS
\]
if and only if 
\[
\bar{\sigma}(\hat{G}_1(\lambda)) < \frac{1}{\hat{\gamma}(\lambda)} \quad \forall \lambda \in \pS.
\]
\end{proposition}

\begin{proof}
The claim follows from Theorem~\ref{thm:inf_dim} by taking $\Pi = \TwoTwo{\gamma I}{0}{0}{-I}$ and $\Psi = \TwoTwo{\psi_1I}{0}{0}{I}$.
\end{proof}

To recover the standard small-gain theorem, simply take $\psi_1 = 1$.

\subsection{Frequency-weighted passivity theorem}

The next result is a generalization of the well-known passivity theorem.

\begin{proposition}
  Let $\theta : \pS \to (-\frac{\pi}{2}, \frac{\pi}{2})$ be a continuous function and $\theta(\lambda^*) = -\theta(\lambda)$ for all $\lambda \in
  \pS$. Then given a square $\hat{G}_1 \in \RHinf$, $[G_1, G_2]$ is uniformly stable over all $\hat{G}_2 \in \RHinf$ satisfying
\[
e^{\jj\theta(\lambda)} \hat G_2(\lambda) + e^{-\jj\theta(\lambda)} \hat G_2(\lambda)^* \geq 0 \quad \forall \lambda \in \pS
\] 
if and only if 
\[
e^{-\jj\theta(\lambda)} \hat G_1(\lambda) + e^{\jj\theta(\lambda)} \hat G_1(\lambda)^* < 0 \quad \forall \lambda \in \pS. 
\]
\end{proposition}

\begin{proof}
  The claim follows from Theorem~\ref{thm:inf_dim} by taking by taking 
\[
\Pi = \TwoTwo{0}{e^{\jj\theta} I}{e^{-\jj\theta} I}{0} \text{ and } 
  \Psi = \frac{1}{\sqrt{2}}\TwoTwo{e^{\jj\theta_1}I}{e^{\jj\theta_2}I}{e^{\jj\theta_1}I}{-e^{\jj\theta_2}I}, 
\]
where $\theta_1, \theta_2 \in \mathcal{H}_\infty \cap \mathcal{C}$ are any two functions that satisfy $\theta_1(\lambda^*) = -\theta_1(\lambda)^*$,
$\theta_2(\lambda^*) = -\theta_2(\lambda)^*$, and also $\theta_2(\lambda) - \theta_1(\lambda)^* = \theta(\lambda)$ for all $\lambda \in \pS$.
\end{proof}

To recover the standard passivity theorem, simply take $\theta = 0$.






\section{Conclusions} \label{sec: conclusion}

This paper established multiple versions of converse integral quadratic constraint (IQC) results within the linear time-invariant setting. They
involve both closed-loop stability and uniform closed-loop stability, in conjunction with various requirements on the multipliers defining the
corresponding IQCs. These results corroborate the utility of IQCs in robustness analysis by demonstrating that such analysis is not conservative
provided that the feedback system is required to be robustly stable against all uncertainties described by a certain IQC. The IQC results were then
specialized to derive several converse passivity theorems for multivariable transfer functions, which have implications in control systems interacting
with unknown but passive environment (e.g. robotics). Generalized small-gain and passivity theorems with frequency weighting functions were also
established based on an extension of a converse IQC result.

Future work may involve seeking converse results for linear time-varying state-space systems and large-scale interconnected networks. Converse results
on classes of negative imaginary systems~\cite{Ang06, PetLan10} and systems manifesting mixed small-gain, passivity, and negative imaginariness across
frequencies are also worth investigating. Examining converse IQC conditions that only hold on segments of the frequency axis in the spirit of the
generalized Kalman-Yakubovich-Popov lemma~\cite{IwaHar95} is another interesting direction.

\section{Acknowledgement}
\label{sec:ack}
The authors would like to express their gratitude to Prof. dr. Arjan van der Schaft of the University of Groningen, the Netherlands, 
for his encouragement and constructive comments that helped better this manuscript.

\input{Addendum_V4_add}

\end{document}

%% file: Addendum_V4_add.tex
\vspace{0.5cm}
\noindent
{\large\bf Addendum to ``Converse Theorems for Integral Quadratic Constraints''}

\vspace{0.25cm}

\begin{abstract}
Consider a linear time-invariant feedback system consisting of two open-loop stable subsystems. In this note, 
it is established that if such feedback system is stable for one subsystem being any arbitrarily passive system, then
the other subsystem must be output strictly passive.
\end{abstract}

\emph{keywords: } 
  robust stability, integral quadratic constraints, passivity, small-gain.

\section{Introduction}

This addendum is intended to fill in the missing entry from Table 1 of~\cite{KhoKao21} (which also appears in 
the conference paper~\cite{KKS20}). Specifically, the bold alphabet \textbf{N} contained in
the table below is novel and established herein. It corresponds to the result that the feedback system $[-G_1, G_2]$ being stable for all stable passive
$G_2$ implies $G_1$ is output strictly passive. In~\cite{KKS20}, it is already noted that such necessity condition holds when $G_1$ and $G_2$ are 
single-input-single-output. Here we show that the condition holds in the general multiple-input-multiple-output case. 

\begin{table}[h]
\begin{center}
\begin{tabular}{|l|c|c|c|}
  \hline
  & $G_1\in\mathcal{P}$   & $G_1\in\mathcal{P}_O$          & $G_1\in\mathcal{P}_I$  \\
  \hline 
  $\forall G_2\in\mathcal{P}$   & $\mathrm{N}\cancel{\mathrm{S}}$           &   $\boldsymbol{\mathrm{N}}\mathrm{S}$        & $\cancel{\mathrm{N}}\mathrm{S}$   \\
  \hline
  $\forall G_2\in\mathcal{P}_O$ & $\mathrm{N}\mathrm{S}$ & $\cancel{\mathrm{N}}\mathrm{S}$ & $\cancel{\mathrm{N}}\mathrm{S}$     \\
  \hline
  $\forall G_2\in\mathcal{P}_I$ & $\mathrm{N}\mathrm{S}$ & $\cancel{\mathrm{N}}\mathrm{S}$ & $\cancel{\mathrm{N}}\mathrm{S}$  \\
  \hline
\end{tabular}
\end{center}
\caption{Conditions on robust stability of $[-G_1,G_2]$ over a set of $G_2$.}
{\footnotesize
N~/~$\cancel{\mathrm{N}}$: the condition in the top row is (necessary~/~not necessary) for robust stability over the set in the first column.\\
S~/~$\cancel{\mathrm{S}}$: the condition in the top row is (sufficient~/~not sufficient) for robust stability over the set in the first column.
}
\label{tab11}
\end{table}

\section{Notation and Preliminaries}

The notation and terminology used in this note is more or less standard in systems theory literature, and exactly the same as in~\cite{KhoKao21}. Here
we simply mention those not covered in~~\cite{KhoKao21}.
A matrix $M \in \Complex^{n \times n}$ is said to be passive if its Hermitian part is positive semi-definite; i.e., $M + M^* \geq 0$. Moreover, $M$ is 
output strictly passive if there exists $\epsilon > 0$ such that $M + M^* \geq \epsilon M^*M$. The rank of $M$ is denoted as $\mathrm{rank}(M)$.

Given matrices
$M \in \Complex^{n \times n}$ and $N \in \Complex^{n \times n}$, $M$ is said to be \emph{Hermitian congruent} to $N$ if there exists non-singular $T \in \Complex^{n \times n}$ such that $M = TNT^*$. If, in addition, $M$, $N$, and $T$ are real, then $M$ is said to be \emph{congruent} to $N$. 
Given $k$-tuple $\theta^{(k)}:=(\theta_1,\theta_2,\cdots,\theta_k)$, we define 
\[
U(\theta^{(k)}) := \begin{bmatrix}
e^{j \theta_1} & & 0 \\
& \ddots & \\
0 & & e^{j \theta_{k}} 
\end{bmatrix}
\]
and
\[
V(\theta^{(k)}) := \begin{bmatrix}
\STwoTwo{\cos \theta_1}{\sin \theta_1}{-\sin \theta_1}{\cos \theta_1} & & 0 \\
& \ddots & \\
0 & &  \STwoTwo{\cos \theta_{k}}{\sin \theta_{k}}{-\sin \theta_{k}}{\cos \theta_{k}}
\end{bmatrix}.
\]
We have the following lemmas. 
\begin{lemma}[\cite{DokIkr02}]
\label{lem:1}
Given (real) matrices $M \in \Real^{n \times n}$ and $N \in \Real^{n \times n}$, if $M$ and $N$ are Hermitian congruent, then they are also congruent.
\end{lemma}

\begin{lemma}[\cite{JohFur01, FurJoh03}]
\label{lem:2}
Let $M \in \Complex^{n \times n}$.  $M$ is passive if and only if $M$ is Hermitian congruent to
\begin{align}
N_{\rm C} := 
\begin{bmatrix}
U(\theta^{(k)}) & 0 & 0 \\
0 & D_{\ell} & 0 \\
0 & 0 & \mathbf{0}_{n-k-2\ell}
\end{bmatrix},
\label{eq:NC}
\end{align}
where $-\frac{\pi}{2} \leq \theta_{k} \leq \dots \leq \theta_1 \leq \frac{\pi}{2}$, $\ell$ is a nonnegative integer satisfying
$k+2\ell = \mathrm{rank}(M)$, $D_{\ell}$ is a direct sum of $\ell$ copies of the block $\STwoTwo{1}{2}{0}{1}$, and $\mathbf{0}_{n-k-2\ell}$ is
$(n-k-2\ell)\times(n-k-2\ell)$ zero matrix.
\end{lemma}

\begin{lemma} \label{lem: not_OSP}
$M \in \Complex^{n \times n}$ is passive but not output strictly passive if and only if it is 
Hermitian congruent to the $N_{\rm C}$ matrix defined in~\eqref{eq:NC} with 
at least one of the following conditions satisfied: $\ell > 0$, $\theta_1 = \frac{\pi}{2}$, $\theta_k = -\frac{\pi}{2}$.
\end{lemma}

\begin{proof}
It follows from Lemma~\ref{lem:2} that $M$ is passive if and only if $M + M^*$ is Hermitian congruent to
\[
N_{\rm C}+N_{\rm C}^* = \begin{bmatrix}
U(\theta^{(k)}) + U(\theta^{(k)})^* & 0 & 0 \\
0 & F_{\ell} & 0 \\
0 & 0 & \mathbf{0}_{n-k-2\ell}
\end{bmatrix},
\]
where $F_\ell$ is a direct sum of $\ell$ copies of block $\STwoTwo{2}{2}{2}{2}$. Let $r := \mathrm{rank}(M) = \mathrm{rank}(M^*M) = k + 2\ell$. Note
that $\mathrm{rank}(N_{\rm C}+N_{\rm C}^*) = \mathrm{rank}(M + M^*) < r$ if and only if at least one of the following conditions holds: 
$\ell > 0$, $\theta_1 = \frac{\pi}{2}$, and $\theta_k = -\frac{\pi}{2}$. Moreover, $\mathrm{rank}(M + M^*) < r$ 
is equivalent to that there exists some non-zero vector $v$ such that $Mv \not = 0$ and $(M+M^*)v =0$, which in turn means that 
there can be no $\epsilon > 0$ such that $M + M^* \geq \epsilon M^*M$. The last statement is equivalent to $M$ being not output strictly passive.
\end{proof}

\begin{lemma} \label{lem: real_congr}
A passive (real) $M \in \Real^{n \times n}$ is congruent to
\begin{align}
\label{eq:NR}
N_{\rm R}:=\begin{bmatrix}
V(\theta^{(k)}) & 0 & 0 \\
0 & D_\ell & 0 \\
0 & 0 & \mathbf{0}_{n - 2k - 2\ell}
\end{bmatrix}
\end{align}
where $0 \leq \theta_{k} \leq \dots \leq \theta_1 \leq \frac{\pi}{2}$, $D_{\ell}$ is a direct sum of $\ell \geq 0$ copies of the block
$\STwoTwo{1}{2}{0}{1}$, and $2k + 2\ell = \mathrm{rank}(M)$. Moreover, $M$ is passive but not output strictly passive if and only if 
it is congruent to the $N_{\rm R}$ matrix defined in~\eqref{eq:NR} with 
at least one of the following conditions satisfied: $\ell > 0$ , $\theta_1 = \frac{\pi}{2}$.
\end{lemma}

\begin{proof}
By Lemma~\ref{lem:2}, $M$ is Hermitian congruent to $N_{\rm C}$ defined in (\ref{eq:NC}). Since $M$ is real, $e^{j\theta}$ is an eigenvalue 
of $M$ if and only if $e^{-j\theta}$ is. Moreover, since $\STwoTwo{e^{j\theta}}{0}{0}{e^{-j\theta}}$ is Hermitian congruent to $\STwoTwo{\cos \theta}{\sin \theta}{-\sin \theta}{\cos \theta}$, we can further conclude that $M$ is Hermitian congruent to $N_{\rm R}$. Finally, by 
Lemma~\ref{lem:1}, we conclude that $M$ is congruent to $N_{\rm R}$. The second part of the claim follows the same arguments as in Lemma~\ref{lem: not_OSP}.
\end{proof}

\section{Main result}

\begin{proposition}
Consider the feedback system $[-G_1,G_2]$ shown in Figure~1 of~\cite{KhoKao21}, where the subsystems $G_1$ and $G_2$ are square and stable.
It holds that $[-G_1,G_2]$ is stable for any passive $G_2$ only if $G_1$ is output strictly passive.
\end{proposition}

\begin{proof}
  First note that by~\cite[Rem. 6]{KhoKao21}, $G_1$ is necessarily passive.  Suppose to the contrapositive that $G_1$ is passive but not output
  strictly passive, i.e. there exists no $\epsilon > 0$ such that
  $\hat{G}_1(\lambda) + \hat{G}_1(\lambda)^* \geq \epsilon \hat{G}_1(\jw)^*\hat{G}_1(\jw)$ for all $\jw\in\pS$. This means there exists
  $\lambda_0\in\pS$ for which $\hat{G}_1(\lambda_0)$ is passive but not output strictly passive.

  If $\hat{G}_1(\lambda_0) \in \Real^{n \times n}$ (is real), then by Lemma~\ref{lem: real_congr} there exists nonsingular $T \in \Real^{n \times n}$
  such that
  \[
  \hat{G}_1(\lambda_0) = T \begin{bmatrix}
V(\theta^{(k)}) & 0 & 0 \\
0 & D_\ell & 0 \\
0 & 0 & \mathbf{0}_{n - 2k - 2\ell}
\end{bmatrix} T^T.
  \]
By Lemma~\ref{lem: real_congr}, we have $\theta_1=\frac{\pi}{2}$ or $\ell >0$, or both. In any event, 
let $\hat{G}_2$ be the following (real) matrix (in case $\ell=0$, $E_{\ell}$ is removed)
   \[
  \hat{G}_2 := T^{-T} \begin{bmatrix}
V(\theta^{(k)}) & 0 & 0 \\
0 & E_\ell & 0 \\
0 & 0 & \mathbf{0}_{n - 2k - 2\ell}
\end{bmatrix}  T^{-1},
  \]
where $E_{\ell}$ is a direct sum of $\ell$ copies of the block
\[
\TwoTwo{1}{0}{-2}{1}.
\]
One can readily verify that $\hat{G}_2$ is passive and and $\det(I + \hat{G}_2\hat{G}_1(\lambda_0)) = 0$, i.e. $[-G_1, G_2]$ is unstable.

In contrast, if $\hat{G}_1(\lambda_0) \in \Complex^{n \times n}$ is not a real matrix, then by Lemma~\ref{lem:2} there exists nonsingular
$T \in \Complex^{n \times n}$ such that
  \[
  \hat{G}_1(\lambda_0) = T \begin{bmatrix}
U(\theta^{(k)}) & 0 & 0 \\
0 & D_{\ell} & 0 \\
0 & 0 & \mathbf{0}_{n-k-2\ell}
\end{bmatrix} T^*.
  \]
  Again by Lemma~\ref{lem: not_OSP}, at least one of the conditions: $\theta_1 = \frac{\pi}{2}$, $\theta_k = -\frac{\pi}{2}$, $\ell > 0$, holds. Let 
\[
X := T^{-*} \begin{bmatrix}
U(\theta^{(k)}) & 0 & 0 \\
0 & E_{\ell} & 0 \\
0 & 0 & \mathbf{0}_{n-k-2\ell}
\end{bmatrix} T^{-1}.
\]
Note that $X$ is passive and $\det(I + X\hat{G}_1(\lambda_0)) = 0$. The remainder of the proof is dedicated to constructing a passive, stable $G_2$
such that $\hat{G}_2(\lambda_0) = X$, in which case $[-G_1, G_2]$ is unstable.

By \cite[Lem. 1]{KKS20}, $I + X$ is nonsingular and $Y := (I + X)^{-1}(I - X)$ satisfies $\bar{\sigma}(Y) = 1$. Using \cite[Lem. 1.14]{Vin01}, one may
construct a $\hat{Q} \in \RHinf$ such that $\hat{Q}(\lambda_0) = Y$, $\|\hat{Q}\|_\infty = 1$, and $\bar{\sigma}(\hat{Q}(\lambda)) < 1$ for all
$\lambda \in \pS \setminus \{\lambda_0, \lambda_0^*\}$. Note that $X = (I - Y)(I + Y)^{-1}$, thus $(I + \hat{Q})^{-1}(\lambda)$ is well-defined at $\lambda = \lambda_0$, and together with the fact $\bar{\sigma}(\hat{Q}(\lambda)) < 1$ for all
$\lambda \in \pS \setminus \{\lambda_0, \lambda_0^*\}$ these imply that
$(I + \hat{Q})^{-1} \in \RHinf$. Now let $\hat{G}_2 := (I - \hat{Q})(I + \hat{Q})^{-1}$, then $G_2$ is stable, passive \cite[Lem. 2]{KKS20}, and
$\hat{G}_2(\lambda_0) = X$. 
\end{proof}

\section{Acknowledgement}

Useful discussions with Chao Chen and Di Zhao are gratefully acknowledged.


%% file: IQC_necessity_arXiv_v2.bbl
\begin{thebibliography}{10}

\bibitem{Ang06}
D.~Angeli.
\newblock Systems with counterclockwise input-output dynamics.
\newblock {\em IEEE Trans. Autom. Contr.}, 51(7):1130--1143, 2006.

\bibitem{AMP16}
M.~Arcak, C.~Meissen, and A.~Packard.
\newblock {\em Networks of Dissipative Systems: Compositional Certification of
  Stability, Performance, and Safety}.
\newblock Springer, 2016.

\bibitem{BaoLee07}
J.~Bao and P.~L. Lee.
\newblock {\em Process Control: The Passive Systems Approach}.
\newblock Advances in Industrial Control. Springer, 2007.

\bibitem{colgate-hogan}
J.~E. Colgate and N.~Hogan.
\newblock Robust control of dynamically interacting systems.
\newblock {\em International Journal of Control}, 48(1):65--88, 1988.

\bibitem{CurZwa95}
R.~F. Curtain and H.~J. Zwart.
\newblock {\em An Introduction to Infinite-Dimensional Linear Systems Theory}.
\newblock Texts in Applied Mathematics 21. Springer-Verlag, 1995.

\bibitem{GGLD90}
M.~Green, K.~Glover, D.~Limebeer, and J.~C. Doyle.
\newblock A {$J$-spectral} factorization approach to $\mathscr{H}_{\infty}$
  control.
\newblock {\em SIAM J. Control Optim.}, 28:1350--1371, 1990.

\bibitem{GreLim95}
M.~Green and D.~J.~N. Limebeer.
\newblock {\em Linear Robust Control}.
\newblock Information and System Sciences. Prentice-Hall, 1995.

\bibitem{IwaHar95}
T.~Iwasaki and S.~Hara.
\newblock Generalized {KYP} lemma: Unified frequency domain inequalities with
  design applications.
\newblock {\em IEEE Trans. Autom. Contr.}, 50(1):41--59, 1995.

\bibitem{IwaHar98}
T.~Iwasaki and S.~Hara.
\newblock Well-posedness of feedback systems: Insights into exact robustness
  analysis and approximate computations.
\newblock {\em IEEE Trans. Autom. Contr.}, 43(5):619--630, 1998.

\bibitem{KKS_CDC19}
C.-Y. Kao, S.~Z. Khong, and A.~van~der Schaft.
\newblock On the converse passivity theorems for {LTI} systems.
\newblock {\em submitted to 58th IEEE Conference on Decision and Control},
  2019.

\bibitem{Kho11}
S.~Z. Khong.
\newblock {\em Robust stability analysis of linear time-varying feedback
  systems}.
\newblock PhD thesis, The University of Melbourne, 2011.

\bibitem{KLR16}
S.~Z. Khong, E.~Lovisari, and A.~Rantzer.
\newblock A unifying framework for robust synchronisation of heterogeneous
  networks via integral quadratic constraints.
\newblock {\em IEEE Trans. Autom. Contr.}, 2016.
\newblock In press.

\bibitem{KhoSch18}
S.~Z. Khong and A.~van~der Schaft.
\newblock On the converse of the passivity and small-gain theorems for
  input-output maps.
\newblock {\em Automatica}, 97:58--63, 2018.

\bibitem{Kim97}
H.~Kimura.
\newblock {\em Chain-scattering approach to $\boldsymbol{H}_{\infty}$ control}.
\newblock Birkhauser, Boston, MA, USA, 1997.

\bibitem{KMXGA14}
N.~Kottenstette, M.~J. McCourt, M.~Xia, V.~Gupta, and P.~J. Antsaklis.
\newblock On relationships among passivity, positive realness, and
  dissipativity in linear systems.
\newblock {\em Automatica}, 50:1003--1016, 2014.

\bibitem{LBEM13}
R.~Lozano, B.~Brogliato, O.~Egeland, and B.~Maschke.
\newblock {\em Dissipative Systems Analysis and Control: Theory and
  Applications}.
\newblock Springer, 2013.

\bibitem{MJKR10}
A.~Megretski, U.~J\"{o}nsson, C.-Y. Kao, and A.~Rantzer.
\newblock {\em The Control Handbook}, chapter Integral quadratic constraints.
\newblock Second edition, 2010.

\bibitem{MegRan97}
A.~Megretski and A.~Rantzer.
\newblock System analysis via integral quadratic constraints.
\newblock {\em IEEE Trans. Autom. Contr.}, 42(6):819--830, 1997.

\bibitem{MegTre93}
A.~Megretski and S.~Treil.
\newblock Power distribution inequalities in optimization and robustness of
  uncertain systems.
\newblock {\em J. Math. Syst., Estimat. Control}, 3(3):301--319, 1993.

\bibitem{PetLan10}
I.~R. Petersen and A.~Lanzon.
\newblock Feedback control of negative imaginary systems.
\newblock {\em IEEE Control System Magazine}, 30(5):54--72, 2010.

\bibitem{Str15}
S.~Stramigioli.
\newblock Energy-aware robotics.
\newblock In M.~K. Kamlibel, A.~A. Julius, R.~Pasumarthy, and J.~M.~A.
  Scherpen, editors, {\em Mathematical Control Theory {I}: Nonlinear and Hybrid
  Control Systems}, Lecture Notes in Control and Information Sciences,
  chapter~3, pages 37--50. Springer, 2015.

\bibitem{Sch17}
A.~van~der Schaft.
\newblock {\em $L_2$-Gain and Passivity Techniques in Nonlinear Control}.
\newblock Springer, 3rd ($1$st edition 1996, $2$nd edition 2000) edition, 2017.

\bibitem{Vid81}
M.~Vidyasagar.
\newblock {\em Input-Output Analysis of Large-Scale Interconnected Systems}.
\newblock Springer-Verlag, 1981.

\bibitem{Wil72}
J.~C. Willems.
\newblock Dissipative dynamical systems part {I}: General theory and part {II}:
  Linear systems with quadratic supply rates.
\newblock {\em Arch. Rational Mechanics Analysis}, 45(5):321--393, 1972.

\bibitem{ZDG96}
K.~Zhou, J.~C. Doyle, and K.~Glover.
\newblock {\em Robust and Optimal Control}.
\newblock Prentice-Hall, Upper Saddle River, NJ, 1996.

\end{thebibliography}

\begin{thebibliography}{1}
\providecommand{\url}[1]{#1}
\csname url@samestyle\endcsname
\providecommand{\newblock}{\relax}
\providecommand{\bibinfo}[2]{#2}
\providecommand{\BIBentrySTDinterwordspacing}{\spaceskip=0pt\relax}
\providecommand{\BIBentryALTinterwordstretchfactor}{4}
\providecommand{\BIBentryALTinterwordspacing}{\spaceskip=\fontdimen2\font plus
\BIBentryALTinterwordstretchfactor\fontdimen3\font minus
  \fontdimen4\font\relax}
\providecommand{\BIBforeignlanguage}[2]{{%
\expandafter\ifx\csname l@#1\endcsname\relax
\typeout{** WARNING: IEEEtran.bst: No hyphenation pattern has been}%
\typeout{** loaded for the language `#1'. Using the pattern for}%
\typeout{** the default language instead.}%
\else
\language=\csname l@#1\endcsname
\fi
#2}}
\providecommand{\BIBdecl}{\relax}
\BIBdecl

\bibitem{KhoKao21}
S.~Z. Khong and C.-Y. Kao, ``Converse theorems for integral quadratic
  constraints,'' \emph{IEEE Transactions on Automatic Control}, 2020. 
  DOI: 10.1109/TAC.2020.3024269. Available on IEEE Xplore.

\bibitem{KKS20}
C.-Y. Kao, S.~Z. Khong, and A.~van~der Schaft, ``On the converse passivity
  theorems for {LTI} systems,'' in \emph{Proceedings of 2020 IFAC World Congress}, 
  Berlin, Germany, 2020.

\bibitem{DokIkr02}
D.~Z. Dokovi\'{c} and K.~D. Ikramov, ``On the congruence of square real
  matrices,'' \emph{Linear Algebra and Its Applications}, vol. 353, pp.
  149--158, 2002.

\bibitem{JohFur01}
C.~R. Johnson and S.~Furtado, ``A generalization of {Sylvester's} law of
  inertia,'' \emph{Linear Algebra and Its Applications}, vol. 338, pp.
  287--290, 2001.

\bibitem{FurJoh03}
S.~Furtado and C.~R. Johnson, ``Spectral variation under congruence for a
  nonsingular matrix with $0$ on the boubound of its field of values,''
  \emph{Linear Algebra and Its Applications}, vol. 359, pp. 67--78, 2003.

\bibitem{Vin01}
G.~Vinnicombe, \emph{Uncertainty and Feedback --- $\mathscr{H}_{\infty}$
  loop-shaping and the $\nu$-gap metric}.\hskip 1em plus 0.5em minus
  0.4em\relax London: Imperial College Press, 2001.

\end{thebibliography}
